\documentclass{amsart}
\usepackage{bbm}
\usepackage{enumitem}
\usepackage{amsmath}

\usepackage{subcaption}
\usepackage{algorithm}
\usepackage{algorithmic}
\usepackage{tikz}
\usetikzlibrary{matrix}
\tikzset{sqvec/.style={matrix, thick}}
\usepackage[square,numbers]{natbib}
\usepackage{hyperref}
\hypersetup{
    colorlinks=true,
    allcolors=blue,
    pdftitle={Exact recovery for seeded graph matching},
    pdfpagemode=FullScreen,
}

\title{Exact recovery for seeded graph matching}
\author{Nicolas Fraiman, Michael Nisenzon}
\date{\today}

\newcommand{\Ex}[1]     {\mathbb{E}\left[#1\right]}
\newcommand{\Var}[1]    {\mathbb{V}\left[#1\right]}
\newcommand{\Prob}[2][] {\mathbb{P}_{#1}\left(#2\right)}
\newcommand{\Ind}[1]    {\mathbbm{1}_{#1}}

\newcommand{\expt}[1] {\exp\left( #1 \right)}
\newcommand{\Nbr}[2][]    {\mathcal{N}_{#1}\left(#2\right)}

\newcommand{\eps}   {\varepsilon}

\newcommand{\I}    {\mathcal{I}}
\newcommand{\J}    {\mathcal{J}}
\newcommand{\U}    {\mathcal{U}}
\newcommand{\R}    {\mathcal{R}}
\newcommand{\C}    {\mathcal{C}}
\newcommand{\Score}    {\mathcal{S}}
\newcommand{\ddegA}{\deg_{A,\U}}
\newcommand{\ddegB}{\deg_{B,\U}}

\DeclareMathOperator{\sgn}{sign}

\DeclareMathOperator{\scsbm}{SCSBM}

\DeclareMathOperator{\Bern}{Bern}
\DeclareMathOperator{\Bin}{Bin}

\theoremstyle{plain}
\newtheorem{theorem}{Theorem}
\newtheorem{lemma}{Lemma}

\theoremstyle{remark}

\theoremstyle{definition}
\newtheorem{definition}{Definition}

\allowdisplaybreaks
\begin{document}
\begin{abstract}
We study graph matching between two correlated networks in the almost fully seeded regime, where all but a vanishing fraction of vertex correspondences are revealed. Concretely, we consider the correlated stochastic block model and assume that $|\U| = n^{1-\alpha}$ vertices remain unrevealed for some $\alpha \in (0,1)$, while the remaining $n - n^{1-\alpha}$ vertices are provided as seed correspondences. Our goal is to determine when the true permutation can be recovered efficiently as the proportion of unrevealed vertices vanishes.

We prove that exact recovery of the remaining correspondences is achievable in polynomial time whenever $\lambda s^{2} > 1 - \alpha$, where $\lambda = (a+b)/2$ is the SBM density parameter and $s$ denotes the edge retention parameter. This condition smoothly interpolates between the fully seeded setting and the classical unseeded threshold $\lambda s^{2} > 1$ for matching in correlated Erd\H{o}s--R\'enyi graphs. Our analysis applies to both a simple neighborhood--overlap rule and a bistochastic relaxation followed by projection, establishing matching achievability in the almost fully seeded regime without requiring spectral methods or message passing.

On the converse side, we show that below the same threshold, exact recovery is information-theoretically impossible with high probability. Thus, to our knowledge, we obtain the first tight statistical and computational characterization of graph matching when only a vanishing fraction of vertices remain unrevealed. Our results complement recent progress in semi-supervised community detection, such as \cite{Nis24}, by demonstrating that revealing all but $n^{1-\alpha}$ correspondences similarly lowers the information threshold for graph matching.
\end{abstract}

\maketitle

\section{Introduction}

Graph matching is the task of recovering an unrevealed bijection between the vertex sets of two related networks, and arises in numerous domains including social network de-anonymization, biological network alignment, multimodal knowledge graph integration, and multi--view learning. The canonical setting involves two graphs that are noisy or anonymized views of a shared structure, and the goal is to infer both the true correspondence and, in many applications, true community structure across the graphs.

A widely studied model for this problem is the correlated Erd\H{o}s--R\'enyi model, in which each observed graph $A$ and $B$ is obtained by independently subsampling edges from a common parent graph with edge correlation $s \in [0,1]$. Classical results show that exact recovery in the unseeded setting---where no vertex correspondences are revealed---is possible if and only if the expected degree of the intersection graph exceeds the graph connectivity threshold. Precisely, Cullina and Kiyavash~\cite{CK17} established that exact matching succeeds if $\lambda s^2 > 1$, where $\lambda \log n$ is the average degree and $n$ is the number of vertices, while below this threshold the true permutation becomes statistically indistinguishable due to automorphisms.

Recent work has sharpened these information-theoretic boundaries. Wu, Xu, and Yu~\cite{wu2022settling} derived the sharp threshold for exact unseeded matching in correlated Erd\H{o}s--R\'enyi graphs in both the dense and sparse cases. In structured random graphs, R\'acz and Sridhar~\cite{Racz21} resolved the analogous exact threshold for sparse correlated SBMs, showing that exact unseeded matching is possible if and only if the information-theoretic community signal is sufficiently strong. On the algorithmic front, significant progress has been made toward meeting these limits: Barak et al.~\cite{barak2019black} developed the first quasipolynomial-time algorithm capable of operating far from the near-identical regime $s \approx 1$, Ding et al.~\cite{ding2021efficient} gave a polynomial-time degree-profile method for moderate noise, and Mao et al.~\cite{mao2023} achieved polynomial-time exact recovery at constant correlation in sparse graphs via structured subgraph counts. Most recently, Chai and R\'acz~\cite{chai2024matching} obtained the first polynomial-time exact recovery algorithm for correlated SBMs in the full information-theoretic regime. Taken together, these works nearly close the statistical--computational gap for unseeded matching in random graph models.

\subsection{Seeded graph matching}

When a subset of vertex correspondences ("seeds") is provided, matching becomes dramatically easier. Pedarsani and Grossglauser~\cite{pedarsani2011privacy} showed that even a sublinear number of seeds can enable recovery in regimes where unseeded matching is impossible. Yartseva and Grossglauser~\cite{yartseva2013analysis} developed a percolation-based algorithm showing that above a critical seed density, correct matches propagate across the graph. Korula and Lattanzi~\cite{korula2014efficient} gave a greedy seed-expansion procedure with guarantees in correlated Erd\H{o}s--R\'enyi graphs, while Lyzinski et al.~\cite{Lyzinski14} analyzed convex relaxations and highlighted fragility in certain relax-and-round approaches. Empirically, even a handful of seeds can significantly boost matching accuracy: Fishkind et al.~\cite{fishkind2019seeded} demonstrated strong gains in real networks.

Previous work narrowed the gap between information--theoretic and computational requirements in seeded matching. Wang et al.~\cite{Wang21, Wang24} derived sharp information-theoretic thresholds for exact recovery for Erd\H{o}s--R\'enyi graphs with seeds, showing how the minimum seed cardinality scales with sparsity and correlation. They also introduced an efficient algorithm that achieves the threshold. Mossel and Xu~\cite{Xu19} further introduced the large-neighborhood-statistics method, achieving the information-theoretic sparsity limit in polynomial time and showing that as few as $\Omega(\log n)$ seeds in dense graphs (or $n^{3\eps}$ in sparse graphs) suffice for exact recovery.

Seeded graph matching now has methods that are nearly optimal both statistically and computationally in a range of random graph models. A classical approach in the unseeded case solves a relaxation of the quadratic assignment problem, as in the FAQ algorithm~\cite{vogelstein2015fast}, which uses Frank--Wolfe optimization to project onto the Birkhoff polytope. The Seeded Graph Matching (SGM) algorithm~\cite{fishkind2019seeded} integrates seeds directly and produces ranked candidate matches. Yu, Xu, and Lin~\cite{yu2021graph} studied robust matching with partially corrupted seeds, showing that exact recovery remains achievable provided a constant fraction of seeds are correct, and proposing down-weighting strategies. Additional relaxations and kernel formulations have been explored, such as KerGM~\cite{Zhang2019}, which uses kernelized quadratic assignment relaxations to incorporate structural and attribute similarity. Recent deep alignment pipelines build on these ideas but typically trade interpretability and guarantees for supervised flexibility. Seed-supervised GNN architectures like SeedGNN~\cite{yu2023} and embedding-based systems such as DeepLink~\cite{zhou2018} achieve strong empirical performance, yet lack provable recovery guarantees in the regimes studied here.

\subsection{Our contribution}

We study the regime where almost all vertex correspondences are revealed, leaving only a vanishing fraction $|\U| = n^{1-\alpha}$ of unrevealed vertices for some $\alpha \in (0,1)$. In this partially revealed setting, we show that simple neighborhood--overlap--based methods, as well as a bistochastic linear programming relaxation followed by projection, achieve exact recovery of the remaining matches in polynomial time whenever $\lambda s^2 > 1 - \alpha$, where $\lambda = (a+b)/2$ is the average SBM density. While we focus on the almost fully seeded regime $|\U| = n^{1-\alpha}$ for clarity, the threshold $\lambda s^2 > 1-\alpha$ smoothly interpolates between the fully seeded case and the classical unseeded limit $\lambda s^2 > 1$ when $|\U|=\Theta(n)$.

In addition to these sharp statistical and algorithmic guarantees, we provide a comprehensive empirical
evaluation on both synthetic and real networks.
Our experiments confirm the theoretical threshold behavior on correlated stochastic block models, showing a rapid transition to near-exact recovery as the seed fraction increases and the condition $\lambda s^2 > 1-\alpha$ is met. On a diverse collection of real-world graphs---including biological, communication, social, and internet topology networks---we find that even very simple overlap-based methods are competitive with state-of-the-art seeded matching algorithms when a large fraction of correspondences is revealed. These results demonstrate that the almost fully seeded regime is not only theoretically tractable, but also practically robust across heterogeneous graph structures.

Taken together, our results provide the first tight statistical and computational characterization of
graph matching when only a vanishing fraction of vertices remain unrevealed. They complement recent progress in semi-supervised community detection~\cite{Nis24} by showing that revealing all but $n^{1-\alpha}$ correspondences similarly lowers the information threshold for graph matching alone, and that this phenomenon is observable in practice as well as in theory.

\section{Setup and Notation}
We begin by formally defining the \textbf{Seeded Correlated Stochastic Block Model}, denoted by $\scsbm$, which serves as the generative model for the observed graph pair $(A, B)$ throughout this paper.

\begin{definition}[Seeded Correlated SBM]
\label{def:plSBM}
Let $n \in \mathbb{N}$ denote the number of vertices, and fix connectivity parameters
$a,b>0$, a retention parameter $s\in[0,1]$, and a revealed (seed) set $\R\subset[n]$.
The \textbf{Seeded Correlated Stochastic Block Model}
$(A,B)\sim\scsbm(n,a,b,s,\R,\sigma^*,\pi^*)$
is generated as follows:
\begin{enumerate}
\item \textbf{True communities.}
A balanced labeling $\sigma^*:[n]\to\{-1,+1\}$ is fixed, inducing two communities
$\C_1=(\sigma^*)^{-1}(1)$ and $\C_2=(\sigma^*)^{-1}(-1)$ with
$|\C_1|=|\C_2|=n/2$.

\item \textbf{True permutation.}
A bijection $\pi^*:[n]\to[n]$ is sampled uniformly at random and represents the
ground truth vertex correspondence between the two graphs.

\item \textbf{Parent SBM.}
An undirected simple graph $G$ with adjacency matrix
$G\in\{0,1\}^{n\times n}$ is generated as follows.
For each unordered pair $1\le u<v\le n$,
\[
G(u,v)\sim
\begin{cases}
\Bern(a\log n/n), & \sigma^*(u)=\sigma^*(v),\\
\Bern(b\log n/n), & \sigma^*(u)\neq\sigma^*(v),
\end{cases}
\]
with $G(v,u)=G(u,v)$ and $G(u,u)=0$.

\item \textbf{Correlated subsampling.}
Conditioned on $G$, the graphs $A$ and $B'$ are generated independently by
subsampling edges of $G$:
for each $u<v$,
\[
A(u,v)=G(u,v)\cdot X_{uv},\qquad
B'(u,v)=G(u,v)\cdot Y_{uv},
\]
where $X_{uv},Y_{uv}\stackrel{\text{i.i.d.}}{\sim}\Bern(s)$.
Again, $A$ and $B'$ are symmetric with zero diagonals.

\item \textbf{Anonymization.}
The observed graph $B$ is obtained by permuting $B'$ according to $\pi^*$:
\[
B(u,v)=B'((\pi^*)^{-1}(u),(\pi^*)^{-1}(v)).
\]

\item \textbf{Revealed information.}
The algorithm uses the exact correspondence $\pi^*(r)$ for all $r\in\R$,
as well as the community labels $\sigma^*(r)$.
The revealed set $\R$ is assumed to be balanced:
\[
|\C_1\cap\R|=|\C_2\cap\R|.
\]
The set of unrevealed vertices is $\U=[n]\setminus\R$.
\end{enumerate}
The observed data consist of $(A,B)$ together with the restrictions $\pi_\R := \pi^*|_\R$ and $\sigma_\R := \sigma^*|_\R$.
\end{definition}

\begin{definition}[Exact Graph Matching]
Let $(A,B)\sim\scsbm(n,a,b,s,\R,\sigma^*,\pi^*)$.
A graph matching algorithm $\hat{\pi} = \hat{\pi}[A,B, \sigma_\R,\pi_\R]$ takes the generated graphs and partially revealed labels as input and outputs an estimator of the permutation
$\hat{\pi}:[n]\to[n]$.
We say that \textbf{exact recovery} is achieved if
\[
\Prob{\hat{\pi}(u)=\pi^*(u)\ \text{for all }u\in[n]}=1-o(1).
\]
Since $\pi^*$ is known on $\R$, this is equivalent to exact recovery on the unknown $\U$.
\end{definition}

In the fully unseeded case $\R = \emptyset$, prior work \cite{CK17} established that exact graph matching is information-theoretically possible if and only if the edge correlation exceeds the connectivity threshold: $\lambda s^2 > 1$, where $\lambda := (a + b)/2$. This threshold corresponds to the emergence of a giant component in the intersection graph $A \wedge B$.

More recent results ~\cite{Xu19} demonstrate that the presence of even a small seed set $\R$ enables efficient, polynomial-time recovery. In particular, a neighborhood--overlap-based estimator can achieve exact recovery under the same information-theoretic threshold. When the seed set is large---e.g., $|\R| = (1 - o(1))n$---the effective threshold for recovery improves, and the corresponding algorithms become more robust to noise. For seeded, correlated Erd\H{o}s--R\'enyi graphs, Wang et al.~\cite{Wang24} consider the effect of the size of $\R$ on the threshold for exact graph matching. Echoing the bound of~\cite{Nosratinia2018}, they show that exact graph matching is possible at $\lambda s^2 > 1-\alpha$ when $|\U| = n^{1-\alpha}$. Our theorems in the next section quantify these improvements precisely for both graph matching and community recovery tasks in the seeded correlated SBM.

As we consider fixed underlying community labels and permutations, we restrict our analysis to equivariant estimators that preclude the possibility of oracles. We first define equivariance for graph matching and community detection on the correlated SBM as follows:

\begin{definition}[Equivariance] Let $(A,B) \sim \scsbm(n,a,b,s,\R,\sigma^*,\pi^*)$. Let $\Pi_{\U}$ denote the set of bijections $[n]\to[n]$ that act as the identity on $\R$. Then an estimator $\hat{\pi}$ of $\pi^*$ is permutation equivariant if, for any permutation $\tau \in \Pi_\U$ of $\U$,
\[
\hat{\pi}(A, B^{\tau}, \pi_\R,\sigma_\R) = \tau \circ \hat{\pi}(A,B, \pi_\R,\sigma_\R)
\]
where $B^{\tau}$ denotes the matrix obtained by relabeling the vertices in $\U$ by $\tau$ while leaving the vertices in $\R$ fixed and $(\tau \circ \hat{\pi})(u) = \tau(\hat{\pi}(u))$.
In addition, $\hat{\pi}$ is label equivariant if for any relabeling $\ell:\{-1,+1\}\to\{-1,+1\}$ of communities,
\[
\hat{\pi}(A, B, \pi_\R, \sigma_\R^\ell) = \hat{\pi}(A,B,\pi_\R, \sigma_\R)
\]
where we define $\sigma_\R^\ell := \ell \circ \sigma_\R$.
Formally, $\hat{\pi}$ is an equivariant estimator for $\pi^*$ if it is both label and permutation equivariant.

\end{definition}

\section{Description of the Algorithms}
We present a sequence of four algorithms ranging from simple neighborhood--overlap heuristics to convex relaxations and their fast first-order approximations---that all exploit the revealed vertices $\R$. We will show that three of these algorithms achieve exact recovery in the almost fully seeded regime.

We begin with two simple estimators that exploit the revealed vertices $\R$ as anchors. When $|\R|$ is large, each unrevealed vertex $u\in\U$ has a substantial number of edges into $\R$, and the intersection of these neighborhoods across the two graphs becomes highly informative. Algorithms~\ref{alg:seeded_matching},\ref{alg:greedy_matching} formalize this idea by assigning to each pair $(u,v)\in\U\times\U$ a score equal to the number of common neighbors in the revealed set. The resulting score matrix aggregates independent Bernoulli evidence across the seeds and admits a sharp concentration gap between the true match and all false candidates when $\lambda s^{2} > 1-\alpha$. The matching is then obtained either by solving a linear assignment problem using the Hungarian algorithm or by greedy selection.

\begin{algorithm}[t]
\caption{Seeded graph matching via Hungarian algorithm.}
\label{alg:seeded_matching}
\begin{algorithmic}[1]
\STATE \textbf{Input:} adjacency matrices $A,B$ on $[n]$, revealed set $\R\subset[n]$, unrevealed set $\U=[n]\setminus\R$.
\STATE \textbf{Output:} permutation estimate $\hat{\pi}$.
\STATE \textbf{Generate scores:}
For $u,v \in\U$, define $\Nbr[A]{u} = \{r \in \R: A(u,r)=1\}$ and $\Nbr[B]{v} = \{r \in \R: B(v,\pi_\R(r))=1\}$. Set $\Score(u,v) := |\Nbr[A]{u} \cap \Nbr[B]{v}|$ for all $u,v\in\U$.
\STATE \textbf{Apply linear assignment to pick optimal scores:} We apply the Hungarian algorithm \cite{hungarian} to get a matching $\pi_\U$ from the score matrix $\Score$ on the unrevealed nodes $\U$ such that $\pi_\U := \arg\max_{\pi \in \Pi_\U} \sum_{u \in \U}\Score(u,\pi(u))$.
\STATE \textbf{Combine with seeds:} We then define $\hat{\pi}(u) := \pi_\R(u)$ for $u \in \R$ and $\hat{\pi}(u) = \pi_\U(u)$ for $u \in \U$.
\STATE \textbf{Return} $\hat{\pi}$.
\end{algorithmic}
\end{algorithm}

\begin{algorithm}[t]
\caption{Seeded graph matching via greedy algorithm.}
\label{alg:greedy_matching}
\begin{algorithmic}[1]
\STATE \textbf{Input:} adjacency matrices $A,B$ on $[n]$, revealed set $\R\subset[n]$, unrevealed set $\U=[n]\setminus\R$.
\STATE \textbf{Output:} permutation estimate $\hat{\pi}$.
\STATE \textbf{Generate scores:}
For $u,v \in\U$, define $\Nbr[A]{u} = \{r \in \R: A(u,r)=1\}$ and $\Nbr[B]{v} = \{r \in \R: B(v,\pi_\R(r))=1\}$. Set $\Score(u,v) := |\Nbr[A]{u} \cap \Nbr[B]{v}|$ for all $u,v\in\U$.
\STATE \textbf{Greedy selection:} Initialize all $u\in\U$ and $v\in\U$ as unmatched. Consider the multiset of all pairs $(u,v)\in \U\times \U$ with weights $\Score(u,v)$. Process pairs in nonincreasing order of $\Score(u,v)$, and whenever both $u$ and $v$ are currently unmatched, add the match $u\mapsto v$ and mark them matched. Continue until all vertices in $\U$ are matched.
\STATE \textbf{Combine with seeds:} Set $\hat{\pi}(u)=\pi_\R(u)$ for $u\in\R$, and $\hat{\pi}(u)$ equal to the greedy match for $u\in\U$.
\STATE \textbf{Return} $\hat{\pi}$.
\end{algorithmic}
\end{algorithm}

The preceding algorithms are based purely on combinatorial score maximization over a $\U\times\U$ score matrix. We now introduce a convex relaxation that places these heuristics within a broader optimization framework. Algorithm~\ref{alg:lp_matching} formulates seeded graph matching as an $\ell_{1}$ minimization problem over the Birkhoff polytope, enforcing exact agreement on the revealed vertices while allowing fractional matchings on the unrevealed block. The objective $\|AD-DB\|_{1}$ penalizes discrepancies between adjacency patterns under the proposed matching and decomposes into terms involving revealed and unrevealed vertices. This LP can be viewed as a seeded version of standard quadratic assignment relaxations \cite{vogelstein2015fast}, but with an $\ell_{1}$ objective.

\begin{algorithm}[t]
\caption{Seeded graph matching via bistochastic LP + projection}
\label{alg:lp_matching}
\begin{algorithmic}[1]
\STATE \textbf{Input:} adjacency matrices $A,B$ on $[n]$, revealed set $\R\subset[n]$, unrevealed set $\U=[n]\setminus\R$.
\STATE \textbf{Output:} permutation estimate $\hat{\pi}$.
\STATE \textbf{Decision variables:} $D\in\mathbb{R}_{\geq 0}^{n\times n}$ (bistochastic alignment), $Y\in\mathbb{R}_{\geq 0}^{n\times n}$ (slack).
\STATE \textbf{LP objective:} minimize $\|Y\|_{1}=\sum_{i,j} Y_{ij}$.
\STATE \textbf{LP constraints:}
\STATE (i) Row/column sums: $D\mathbf{1}=\mathbf{1}$ and $\mathbf{1}^\top D=\mathbf{1}^\top$.
\STATE (ii) Linearization of $L_1$ residuals: $
Y \geq AD-DB,\quad Y \geq -(AD-DB)$.
\STATE (iii) Seed fixing: 
\[D_{r,i}=\Ind{i=\pi_\R(r)} \ \forall r\in\R,
\qquad
D_{j,\pi_\R(r)}=\Ind{j=r} \ \forall r\in\R.
\]
\STATE \textbf{Solve LP:} obtain an optimal solution $(D^*,Y^*)$ using any LP solver.
\STATE \textbf{Use Hungarian algorithm to project to a permutation on $\U$:}
\STATE (i) Form the weight matrix $W:=D^*_{\U,\U}$.
\STATE (ii)  Compute a maximum-weight matching $\pi_{\U}\in\Pi_{\U}$ by solving
\[
\pi_{\U}\in\arg\max_{\pi\in\Pi_{\U}} \sum_{u\in\U} W_{u,\pi(u)}.
\]
\STATE \textbf{Combine with seeds:} Set $\hat{\pi}(u)=\pi_\R(u)$ for $u\in\R$ and $\hat{\pi}(u)=\pi_{\U}(u)$ for $u\in\U$.
\STATE \textbf{Return} $\hat{\pi}$.
\end{algorithmic}
\end{algorithm}

Although Algorithm~\ref{alg:lp_matching} provides strong theoretical guarantees, its direct implementation requires solving a large-scale linear program, which is computationally prohibitive for very large graphs. To address this, we also consider a fast first-order method that approximately optimizes the same $\ell_{1}$ objective while operating only on the unrevealed block $\U\times\U$. Algorithm~\ref{alg:fw_linear_matching} applies a Frank--Wolfe procedure to the $\ell_{1}$ seeded alignment objective, using the Hungarian algorithm as the linear minimization oracle at each iteration. This approach preserves the seed constraints and avoids forming the entire $n\times n$ decision variable. In Algorithm~\ref{alg:fw_linear_matching}, we work in a canonical vertex ordering in which the revealed set $\R$ precedes the unrevealed set $\U$. We first reindex the adjacency matrix $B$ using the revealed correspondence $\pi_\R$ (leaving the vertices in $\U$ in a default order), and all block matrices are taken with respect to this $\R/\U$ partition. Thus, the algorithm depends only on the observed graphs $A,B$ and the revealed matching $\pi_\R$, and does not require access to the unknown permutation $\pi^*$. Although we do not establish formal convergence guarantees for this variant, our experiments show performance comparable to that of the exact LP solver in the almost fully seeded regime.

\begin{algorithm}[ht]
\caption{Seeded graph matching via Frank--Wolfe $\ell_1$ relaxation}
\label{alg:fw_linear_matching}
\begin{algorithmic}[1]
\STATE \textbf{Input:} adjacency matrices $A,B$ on $[n]$, revealed set $\R\subset[n]$, unrevealed set $\U=[n]\setminus\R$, number of FW iterations $T$.
\STATE \textbf{Canonical ordering:} Relabel vertices so that $\R=\{1,\dots,|\R|\}$ and $\U=\{|\R|+1,\dots,n\}$.
\STATE \textbf{Seed-aligned graph:} Define $\widetilde B := P_{\pi_\R}^{\top} B P_{\pi_\R}$, where $P_{\pi_\R}$ is a permutation on $[n]$ that equals $\pi_\R$ on $\R$ and is the identity on $\U$.
\STATE \textbf{Output:} permutation estimate $\hat{\pi}$.
\STATE \textbf{Decision variables:} $D_{\U,\U}\in\mathbb{R}_{\ge 0}^{|\U|\times|\U|}$. 
\STATE \textbf{Initialization:} initialize $D_{\U,\U}^{(0)}$ to any feasible doubly stochastic matrix (e.g.\ the uniform matrix).

\FOR{$t=0,1,\ldots,T-1$}
    \STATE \textbf{Compute residual blocks (with respect to the $\R/\U$ partition):}
    \begin{align*}
    R_{\R,\U} &= \widetilde B_{\R,\U}D_{\U,\U}^{(t)} - A_{\R,\U},\\
    R_{\U,\R} &= \widetilde B_{\U,\R} - D_{\U,\U}^{(t)}A_{\U,\R},\\
    R_{\U,\U} &= \widetilde B_{\U,\U}D_{\U,\U}^{(t)} - D_{\U,\U}^{(t)}A_{\U,\U}.
    \end{align*}

    \STATE \textbf{$\ell_1$ subgradient blocks:}
    \begin{align*}
    &G_{\R,\U}=\sgn(R_{\R,\U}),\\
    &G_{\U,\R}=\sgn(R_{\U,\R}),\\
    &G_{\U,\U}=\sgn(R_{\U,\U}).
    \end{align*}

\STATE \textbf{Gradient w.r.t.\ the decision block $D_{\U,\U}$:}
\begin{align*}
\nabla f\!\left(D_{\U,\U}^{(t)}\right)
&= \widetilde B_{\R,\U}^\top G_{\R,\U}
   - G_{\U,\R}A_{\U,\R}^\top \\
&\quad + \widetilde B_{\U,\U}^\top G_{\U,\U}
   - G_{\U,\U}A_{\U,\U}^\top .
\end{align*}

    \STATE \textbf{Frank--Wolfe linear minimization:}
    \[
    S^{(t)} \in \arg\min_{S\in\Pi_{\U}}
    \left\langle \nabla f\!\left(D_{\U,\U}^{(t)}\right),\, S\right\rangle.
    \]

    \STATE \textbf{Step size:} $\gamma_t := \frac{1}{t+2}$.

    \STATE \textbf{FW update (unrevealed block):}
    \[
    D_{\U,\U}^{(t+1)} := (1-\gamma_t)D_{\U,\U}^{(t)} + \gamma_t S^{(t)}.
    \]
\ENDFOR

\STATE \textbf{Project to a permutation on $\U$:}
form $W:=D_{\U,\U}^{(T)}$ and compute
\[
\pi_{\U}\in\arg\max_{\pi\in\Pi_{\U}}
\sum_{u\in\U} W_{u,\pi(u)}
\]
using the Hungarian algorithm.

\STATE \textbf{Combine with seeds:}
set $\hat{\pi}(u)=\pi_\R(u)$ for $u\in\R$ and $\hat{\pi}(u)=\pi_{\U}(u)$ for $u\in\U$.

\STATE \textbf{Return} $\hat{\pi}$.
\end{algorithmic}
\end{algorithm}

\section{Proofs of the Main Results}
We provide the threshold for exact recovery in Theorem \ref{th:lower_SBM}. We then show that Algorithms \ref{alg:seeded_matching}, \ref{alg:greedy_matching}, and
\ref{alg:lp_matching} achieve this threshold.  Algorithm~\ref{alg:fw_linear_matching} is a computationally efficient approximation to the LP relaxation, but we do not establish formal convergence or recovery guarantees for it.

\subsection{Lower Bound: Impossibility Below the Threshold}

Previous proofs \cite{CK17} show impossibility by comparing statistically indistinguishable permutations of the isolated nodes
\[
\I := \{u : u \in \U, \sum_{v \in [n]} A(u,v)B(\pi^*(u),\pi^*(v)) = 0\}
\]
i.e., $u$ has no neighbors in the intersection graph under the true alignment.

We simplify this argument in seeded graphs with a useful subset of nodes $\J \subset \I$ that are difficult to match correctly. Intuitively, $\J$ contains vertices whose observed adjacencies are symmetric under within-community swaps. 
\begin{definition}[Hard isolated vertices] A node $u \in \U$ is in $\J$ if
\begin{enumerate}
    \item $u \in \I$.
    \item $A(u,v) = 0$ and $B(\pi^*(u),\pi^*(v)) = 0$ for all $v \in \U$.
    \item For all $w \in \U$ and all $v\in \R$, 
    \[
    A(u,v) = 1 \implies B(w,\pi_\R(v)) = 0.
    \]
\end{enumerate}
    
\end{definition}

We start with a lemma on the overlap scores $\Score(u,v)$ for unrevealed vertices $u,v\in\U$.

\begin{lemma}\label{lem:bad}
Let $(A,B) \sim \scsbm(n,a,b,s,\R,\sigma^*,\pi^*)$ be defined with $|\U| = n^{1-\alpha}$ for some $\alpha \in [0,1]$. Then for any unrevealed $u,v \in \U$ with $v \neq \pi^*(u)$, we have that for any $\eps > 0$, the size of the intersection neighborhood $\Score(u,v)$ satisfies 
\[
\Prob{\Score(u,v) \geq \eps \log n} \leq n^{-\eps\log n + o(\eps\log n)}.
\]
In addition, if $v = \pi^*(u)$, then for any $\eps > 0$,
\[
\Prob{\Score(u,v) \leq \eps \log n} \leq n^{-\lambda s^2 + O(\eps)}.
\]
\end{lemma}

\begin{proof}
Let $p_a = a\log n/n$ and $p_b = b\log n/n$. The score for $u,v$ in distinct communities is stochastically dominated by the score when $u,v$ are in the same community as $(a^2+b^2) > 2ab$. As a result, when $v \neq \pi^*(u)$, $\Score(u,v) \preceq \Bin(|\R|/2, p_a^2s^2) + \Bin(|\R|/2, p_b^2s^2)$. Taking a Chernoff bound for any $t > 0$
\begin{align*}
    \log \Prob{\Score(u,v) > \eps\log n} &\leq \frac{ns^2}{2}(p_a^2+p_b^2)(e^t-1)-\eps t\log (n)\\
    &\leq -\frac{ns^2}{2}(p_a^2+p_b^2)+\eps \log n \\
    & \qquad- \eps\log (n)\log\left(\frac{\eps\log (n)}{ns^2(p_a^2+p_b^2)/2}\right) \\
    &\leq -s^2(a^2+b^2)\log^2(n)/2n+\eps \log n  \\
    &\qquad - \eps\log (n)\log\left(\frac{\eps n}{s^2(a^2+b^2)/2 \log n}\right) \\
    &= -\eps\log^2(n) + o(\eps\log^{2} n)
\end{align*}
where $t$ is optimized at $t^* = \log \left(\frac{\eps\log (n)}{ns^2(p_a^2+p_b^2)/2}\right)$.

The second statement follows from a similar Chernoff bound, where $\Score(u,v) \preceq \Bin(|\R|/2, p_as^2) + \Bin(|\R|/2, p_bs^2)$:
\begin{align*}
    \log \Prob{\Score(u,v) < \eps\log n} &\leq ns^2(p_a+p_b)(e^t-1)/2-\eps t\log (n)\\
    & \leq -ns^2(p_a+p_b)/2+\eps \log n \\
    &\qquad -\eps\log (n)\log\left(\frac{\eps\log (n)}{ns^2(p_a+p_b)/2}\right) \\
    & \leq -\lambda s^2\log(n)+\eps \log n - \eps\log (n)\log\left(\frac{\eps}{\lambda s^2}\right) \\
    & = -\lambda s^2\log(n) + O(\eps\log n)
\end{align*}
where $t$ is evaluated at $t^* = \log \left(\frac{\eps\log (n)}{ns^2(p_a+p_b)/2}\right) = \log(\frac{\eps }{\lambda s^2})$. 
\end{proof}

We use this lemma to derive bounds on the sizes of $\I$ and $\J$.

\begin{lemma}\label{lem:i_size} $\I$ satisfies $|\I| = n^{1-\alpha -\lambda s^2 + o(1)}$ with probability $1-o(1)$.
\end{lemma}
\begin{proof}
    We bound the size of $\I$ using the second moment method. We start with the mean. For $i \in \U$, let $S_i$ be the indicator for node $i$ being isolated in the intersection graph $A \wedge B$ under the true permutation $\pi^*$. In addition, let $S := \sum_{i\in\U}S_i = |\I|$. Then by linearity of expectation
\begin{align*}
\Ex{S} &= \sum_{i \in \U} \Ex{S_i} \\
&= |\U| (1-s^2a\log n/n)^{n/2}(1-s^2b\log n/n)^{n/2} \\
&\geq |\U| \expt{-\frac{s^2a/2 \log n}{1-as^2\log n /n}} \cdot\expt{-\frac{s^2b/2\log n}{1-bs^2\log n /n}} \\
&\geq n^{1-\alpha}\expt{-\frac{\lambda s^2 \log n}{1-a s^2 \log n/n}}\\
&= n^{1-\alpha}n^{-\lambda s^2 \cdot (1+o(1))} 
\end{align*}

We now bound the second moment with exchangeability:
\begin{align*}
    \Ex{S^2} &= \sum_{i \in \U} \Ex{S_i^2} + \sum_{i \neq j \in \U} \Ex{S_iS_j} \leq \Ex{S} + |\U|^2 \Ex{S_iS_j}
\end{align*}
It is then sufficient to bound $\Ex{S_iS_j}$. We note that the probability that $i,j$ are both isolated can be bounded by
\[
\Ex{S_iS_j} \leq (1-s^2a\log n/n)^{n-1}(1-s^2b\log n/n)^{n-1} 
\]
Using the inequality $1+x \leq e^{x}$ gives an upper bound of 
\begin{align*}
\Ex{S_iS_j} &\leq \exp\left(-\frac{n-1}{n}(s^2a\log n+ s^2b\log n)\right) = n^{-2(1+o(1))\lambda s^2}
\end{align*}
We then see that 
\begin{align*}
    \Ex{S^2} &\leq |\U|n^{-\lambda s^2} + |\U|^2n^{-2(1+o(1))\lambda s^2} \\
    &= (1+o(1))\Ex{S}^2
\end{align*}
Applying Chebyshev's inequality lets us conclude
\begin{align*}
\Prob{|S-\Ex{S}| \geq \frac{1}{2}\Ex{S}} &\leq \frac{4(\Ex{S^2}-\Ex{S}^2)}{\Ex{S}^2}\\
&\leq \frac{4((1+o(1))\Ex{S}^2-\Ex{S}^2)}{\Ex{S}^2} \\
&= o(1)
\end{align*}
\end{proof}
\begin{lemma}\label{lem:j_size} $\J$ satisfies $|\J| = n^{1-\alpha -\lambda s^2 + o(1)}$  with probability $1-o(1)$.
\end{lemma}
\begin{proof}
Recall that $\J\subseteq\I$ by definition. By Lemma~\ref{lem:i_size}, $|\I| = n^{1-\alpha-\lambda s^2+o(1)}$ with probability $1-o(1)$. Moreover, for a fixed $u\in\U$, $\Prob{u\in\I}=n^{-\lambda s^2+o(1)}$. Therefore, it suffices to show that  $|\I\setminus\J| = o(|\I|)$ with probability $1-o(1)$. Fix $u\in\U$ and write $v^*=\pi^*(u)$. We bound the probability that $u\in\I \cap \J^c$. Given $u\in\I$, we have
\[
A(u,x)B(v^*,\pi^*(x))=0 \quad \text{for all } x\in[n].
\]
In particular, for any $x\in\U$, if $A(u,x)=1$, then necessarily $B(v^*,\pi^*(x))=0$. Since the edges from $u$ into $\U$ occur with probability $\Theta(\log n/n)$ and $|\U|=n^{1-\alpha}$, a union bound gives
\[
\Prob{\exists x\in\U:\ A(u,x)=1 \mid u\in\I}
\le O(n^{-\alpha}\log n)=o(1).
\]
By symmetry,
\[
\Prob{\exists x\in\U:\ B(v^*,\pi^*(x))=1 \mid u\in\I}=o(1).
\]
Hence, uniformly in $u$, $\Prob{\text{Cond. (2) fails} \mid u\in\I}=o(1)$.

Let $ S_u := \{v\in\R:\ A(u,v)=1\}$ be the seed neighborhood of $u$ in $A$. Condition (3) requires that for every $v\in S_u$ and every $w\in\U$, $B(w,\pi_\R(v))=0$. Conditioned on the revealed information, the events $\{B(w,\pi_\R(v))=0\}$ for distinct edges $(w,\pi_\R(v))$ are independent, with
\[
\Prob{B(w,\pi_\R(v))=1}=\Theta(\log n/n).
\]
Therefore, for some constant $c>0$,
\begin{align*}
\Prob{\text{Cond. (3) holds} \mid S_u}
&= \prod_{v\in S_u}\prod_{w \in \U} \Prob{B(w,\pi_\R(v))= 0} \\
&\ge \exp\left(-c\,|S_u|\,|\U|\,\log(n)/n\right)
\end{align*}
Conditional on $u \in \I$, a Chernoff bound gives $|S_u| = \Theta(\log n)$ with probability $1-o(1)$. Then for some $c_1 > 0$,
\begin{align*}
\Prob{\text{Cond. (3) holds} \mid u\in\I} &\geq \exp\left(-c_1n^{-\alpha}\log^2n\right) = 1-o(1).
\end{align*}
We then have uniformly over $u\in\U$ that
\begin{align*}
\Prob{u\in\I\setminus\J} &\leq \Prob{u\in\I}\cdot\Prob{\text{Cond. (3) fails} \mid u\in\I} = n^{-\lambda s^2+o(1)}\cdot o(1).
\end{align*}
By linearity of expectation,
\begin{align*}
\Ex{|\I\setminus\J|}
&= |\U|\Prob{u\in\I\setminus\J} \\
&= n^{1-\alpha}\cdot n^{-\lambda s^2+o(1)}\cdot o(1)\\
&= o\!\left(n^{1-\alpha-\lambda s^2}\right).
\end{align*}
By Markov, $|\I\setminus\J| = o(|\I|)$ with probability $1-o(1)$. Since $|\J|=|\I|-|\I\setminus\J|$, we conclude that $|\J| = n^{1-\alpha-\lambda s^2+o(1)}$ with probability $1-o(1)$.
\end{proof}

\begin{lemma}[Transposition invariance on $\J$]\label{lem:transposition_invariance}
Fix distinct $u,v\in\U$ with $\sigma^*(u)=\sigma^*(v)$, and let $\tau$ be the transposition
swapping $u$ and $v$ and fixing all other vertices.  Conditional on the revealed information
$(\pi_\R,\sigma_\R)$ and on the event $\{u,v\in\J\}$, we have
\[
(A,B,\pi_\R,\sigma_\R)\ \stackrel{d}{=}\ (A,B^\tau,\pi_\R,\sigma_\R).
\]
\end{lemma}

\begin{proof}
Let $\widetilde B:=P_{\pi_\R}^\top B P_{\pi_\R}$ be the seed-aligned version of $B$. Since $\tau$ fixes $\R$, $\tau$ commutes with this seed-alignment, so it suffices to prove
\begin{align*}
&\Prob{A=A_0,\widetilde B=\widetilde B_0\mid \pi_\R,\sigma_\R,u,v\in\J} = \Prob{A=A_0,\widetilde B=\widetilde B_0^\tau\mid \pi_\R,\sigma_\R,u,v\in\J}.
\end{align*}

Consider the joint distribution of the edge indicators
observable to the estimator: all edges in $A$ and all edges in $\widetilde B$ with at least one endpoint in
$\R\cup\{u,v\}$. Under the model, as $(\sigma^*,\pi^*)$ are fixed, these edge variables are mutually independent across unordered vertex pairs.

We check invariance of each likelihood factor under swapping $u$ and $v$. We first consider edges not incident to $u$ or $v$. If $\{i,j\}\cap\{u,v\}=\emptyset$, then $\widetilde B_0^\tau(i,j)=\widetilde B_0(i,j)$, so the corresponding likelihood factor is unchanged. 

Next, we consider edges between $u$ or $v$ and an unrevealed vertex. For $x\in\U\setminus\{u,v\}$, condition (2) in the definition of $\J$ implies
\[
A(u,x)=A(v,x)=0,\qquad \widetilde B(u,x)=\widetilde B(v,x)=0,
\]
so the likelihood contribution of the pair of edges $\{(u,x),(v,x)\}$ is invariant under $\tau$.

After this, we consider the edge $(u,v)$. This edge is fixed by the transposition, so its likelihood factor is unchanged. 

Finally, we consider the edges between $u$ or $v$ and a revealed vertex. Fix $r\in\R$ and consider the tuple
\[
(A(u,r),A(v,r),\widetilde B(u,r),\widetilde B(v,r)).
\]
Since $\sigma^*(u)=\sigma^*(v)$ and $\sigma^*(r)$ is revealed, the model assigns the same parameters to
edges $(u,r)$ and $(v,r)$. Thus, before conditioning on $u,v\in\J$, this quadruple is exchangeable under
swapping $u$ and $v$.

We now condition on the event $\{u,v\in\J\}$ and distinguish cases according to
the values of $A(u,r)$ and $A(v,r)$:

\begin{enumerate}
\item \textbf{$A(u,r)=A(v,r)=1$.}
By condition (3) in the definition of $\J$, this forces
\[
\widetilde B(u,r)=\widetilde B(v,r)=0,
\]
so the likelihood contribution is unchanged under the transposition $\tau$.

\item $A(u,r)=1$ and $A(v,r)=0$.
Again by condition (3), the event $A(u,r)=1$ implies
\[
\widetilde B(w,r)=0 \quad \text{for all } w\in\U,
\]
and in particular $\widetilde B(u,r)=\widetilde B(v,r)=0$.
Thus the contribution is invariant under swapping $u$ and $v$.

\item $A(u,r)=0$ and $A(v,r)=1$.
This case is symmetric to the previous one and yields the same conclusion.

\item $A(u,r)=A(v,r)=0$.
In this case, condition (3) imposes no further restriction.
However, the event $\{u,v\in\J,\ \sigma^*(u)=\sigma^*(v)\}$ is symmetric in $u$ and $v$,
and the model assigns identical parameters to the edges $(u,r)$ and $(v,r)$.
Therefore, the conditional distribution of
\[
(\widetilde B(u,r),\widetilde B(v,r))
\]
remains exchangeable under swapping $u$ and $v$, and the contribution
is invariant under the transposition $\tau$.
\end{enumerate}

Combining cases (1)--(4) and using independence across unordered pairs, the full conditional likelihood of
$(A,\widetilde B)$ is invariant under $\widetilde B\mapsto\widetilde B^\tau$.
\end{proof}

The resulting lower bound on recovery follows from the lemmas above:
\begin{theorem}\label{th:lower_SBM} Let $|\U| = n^{1-\alpha}$ for some $\alpha \in [0,1]$. If $\lambda s^2 < 1-\alpha$, no equivariant estimator can succeed with probability exceeding $1/2+o(1)$.

\end{theorem}
\begin{proof} We note that when $|\U| = \Theta(n)$ then the bound reduces to the classical case of $\lambda s^2 < 1$ in \cite{CK17}.

Assume $\lambda s^2<1-\alpha$. By Lemma~\ref{lem:j_size}, with probability $1-o(1)$ we have $|\J|\to\infty$. In particular, there exist two distinct vertices $u\neq v$ in the same community contained in $\J$. Let $\tau$ be the within-community transposition swapping $u$ and $v$ and fixing all other vertices. Consider the alternative permutation $\pi' := \tau \circ \pi^*$, which preserves community labels and agrees with $\pi^*$ on $[n]\setminus \{u,v\}$.

Let $\hat\pi$ be any permutation-equivariant estimator of the true permutation $\pi^*$. By equivariance, for every realization over the transposition $\tau$
\[
\hat\pi(A,B^\tau,\pi_\R,\sigma_\R)
= \tau \circ \hat\pi(A,B,\pi_\R,\sigma_\R).
\]
Moreover, by Lemma~\ref{lem:transposition_invariance}  conditional on $(\pi_\R,\sigma_\R)$ and on the event
$\{u,v\in\J,\ \sigma^*(u)=\sigma^*(v)\}$, we have
\[
(A,B,\pi_\R,\sigma_\R) \;\stackrel{d}{=}\; (A,B^\tau,\pi_\R,\sigma_\R).
\]
Using the distributional symmetry above,
\begin{align*}
\Prob{\hat\pi=\pi^*}
&=\Prob{\hat\pi(A,B,\pi_\R, \sigma_\R)=\pi^*} \\
&=\Prob{\hat\pi(A,B^\tau,\pi_\R, \sigma_\R)=\pi^*} \\
&=\Prob{\tau\circ\hat\pi(A,B,\pi_\R, \sigma_\R)=\pi^*} \\
 &=\Prob{\hat\pi=\tau^{-1}\circ\pi^*}\\
 &=\Prob{\hat\pi=\tau\circ\pi^*},
\end{align*}
since $\tau^{-1}=\tau$. The events $\{\hat\pi=\pi^*\}$ and $\{\hat\pi=\tau\circ\pi^*\}$ are
disjoint because $u\neq v$, hence
\[
1 \geq\ \Prob{\hat\pi=\pi^*}+\Prob{\hat\pi=\tau\circ\pi^*}
\ =\ 2\Prob{\hat\pi=\pi^*}.
\]
Therefore $\Prob{\hat\pi=\pi^*}\le 1/2$ on the event $\mathcal{E}$ that a pair $u\neq v\in\J$ exists with $\sigma^*(u) = \sigma^*(v)$. Since $\mathcal{E}$ holds with probability $1-o(1)$ when $\lambda s^2<1-\alpha$, we conclude
\[
\Prob{\hat\pi=\pi^*}\ \le\ \tfrac12 + o(1),
\]
as claimed.
\end{proof}

\subsection{Upper Bound: Algorithmic Achievability}

We first state a bound on the success rate for exact graph matching using the Hungarian algorithm for the $\scsbm$, extending the bounds from \cite{Racz21}:

\begin{theorem}\label{th:upper_SBM}
Let $(A,B) \sim \scsbm(n,a,b,s,\R,\sigma^*,\pi^*)$ be defined with $|\U| = n^{1-\alpha}$ for some $\alpha \in [0,1]$. Then setting $\lambda = \frac{a+b}{2}$, if $\lambda s^2 > 1-\alpha + \eps$ for some $\eps > 0$, Algorithm \ref{alg:seeded_matching} returns an estimator $\hat{\pi}$ in polynomial time satisfying
\[
\Prob{\hat{\pi} = \pi^*} = 1 -o(1).
\]
\end{theorem}

\begin{proof}
Constructing all overlaps $\Score(u,v)=|\Nbr[A]{u}\cap\Nbr[B]{v}|$ takes
$O(|\U|\cdot|\R|)$ time, and the Hungarian algorithm takes $O(|\U|^3)$ time.

Let $\hat{\pi}$ have $k$ misaligned vertices for $2 \leq k \leq |\U|$. As $\hat{\pi}$ is returned by the Hungarian method, necessarily
\begin{equation}
  \sum_{u\in\U}\Score(u,\hat{\pi}(u))
\ge
\sum_{u\in\U}\Score(u,\pi^*(u)).
\label{eq:trace-ineq}  
\end{equation}

Crucially, for each misaligned $u$, the random variables
\[
\Score(u,\hat{\pi}(u)),\qquad \Score(u,\pi^*(u))
\]
are sums of $(1-o(1))n$ Bernoulli variables indexed by the seeds $\R$. Moreover, for distinct misaligned vertices $u_i$, the collections of Bernoulli variables
entering $\Score(u_i,\hat{\pi}(u_i))$ and $\Score(u_i,\pi^*(u_i))$ are disjoint and thus conditionally independent given $(\R,\pi^*)$. Hence each expression in \eqref{eq:trace-ineq} is a sum of $\Theta(k|\R|)$ independent Bernoulli variables  conditional on the revealed set $\R$ and the true permutation $\pi^*$. Let $u_1, \ldots u_k$ be the vertices on which $\hat{\pi}$ and $\pi^*$ disagree.

By a Chernoff bound on $\Theta(k|\R|)$ variables analogous to Lemma~\ref{lem:bad}, the error rate given $k$ misaligned vertices can be bounded for any $\eps' > 0$ by
\begin{align*}
\Prob{\sum_{u\in\U}\Score(u,\hat{\pi}(u))
\ge
\sum_{u\in\U}\Score(u,\pi^*(u))} & = \Prob{\sum_{i=1}^k\Score(u_i,\hat{\pi}(u_i))
\ge
\sum_{i=1}^k\Score(u_i,\pi^*(u_i))} \\
& \leq \Prob{\sum_{i=1}^k\Score(u_i,\hat{\pi}(u_i))> \eps'\log n}\\
&\quad + \ \Prob{\sum_{i=1}^k\Score(u_i,\pi^*(u_i)) < \eps'\log n} \\
& \leq n^{-\eps'\log n + o(\eps' \log n)}+n^{-k\lambda s^2 + o(\eps')}\\
&\quad= n^{-k\lambda s^2 + o(\eps')}\\
\end{align*}

Letting $\eps' \to 0$, gives an upper bound of $n^{-k\lambda s^2 + o(1)}$, allowing us to upper-bound the probability of seeing exactly $k$ errors by 
\[
\Prob{\hat{\pi}\text{ has exactly }k\text{ errors}}
\le
n^{-k\lambda s^{2}+o(1)}.
\]

There are $\binom{|\U|}{k}(k!-1)\le n^{k(1-\alpha)+o(k)}$ such permutations, so
\begin{align*}
\Prob{\hat{\pi}\neq\pi^*}
&\le
\sum_{k\ge 1} n^{k(1-\alpha)} n^{-k\lambda s^{2}+o(k)} \\
&=\sum_{k\ge 1} n^{-k(\lambda s^{2}-(1-\alpha))+o(k)}.
\end{align*}

If $\lambda s^{2}>(1-\alpha)+\eps$ then every term is at most $n^{-k\eps+o(k)}$, and
the series is dominated by its first term:
\[
\Prob{\hat{\pi}\neq\pi^*}=n^{-\eps + o(1)}=o(1).\qedhere
\]
\end{proof}

We now show that the greedy algorithm achieves the threshold for exact recovery in the almost fully seeded case.
\begin{theorem}\label{th:greedy_SBM}
Let $(A,B) \sim \scsbm(n,a,b,s,\R,\sigma^*,\pi^*)$ be defined  with $|\U| = n^{1-\alpha}$ for some $\alpha \in [0,1]$. Then setting $\lambda = \frac{a+b}{2}$, if $\lambda s^2 > (1-\alpha) + \eps$ for some $\eps > 0$, Algorithm \ref{alg:greedy_matching} returns an estimator $\hat{\pi}$ in polynomial time such that
\[
\Prob{\hat{\pi} = \pi^*} = 1 -o(1).
\]
\end{theorem}

\begin{proof}
Computing all scores $\Score(u,v)$ takes $O(|\U|\cdot|\R|)$ time and sorting the multiset
$\{\Score(u,v)\colon u,v\in\U\}$ takes $O(|\U|^2\log|\U|)$, giving total runtime polynomial in~$n$.

Fix any output $\hat{\pi}\neq\pi^*$ and define the error permutation
\[
\rho := \hat{\pi} \circ (\pi^*)^{-1}.
\]
Then $\rho$ decomposes uniquely into $\ell$ disjoint nontrivial cycles,
\[
\rho = C_1\circ\cdots\circ C_\ell,
\qquad 1\leq \ell\leq |\U|/2,
\]
and $u\in\U$ is misaligned iff $u$ lies in one of the cycles.
We now show that each cycle $C_i$ contributes at least one vertex for which the greedy algorithm selects an incorrect pair before the correct pair becomes unavailable.

Fix a cycle $C_i$ and write it as
\[
C_i = (u_1,\ldots,u_k),
\qquad k\geq 2,
\]
so that $\rho(u_j)=u_{j+1}$ for $j=1,\ldots,k-1$ and $\rho(u_k)=u_1$.
Define
\[
v_j := \pi^*(u_j), \qquad \hat{\pi}(u_j)=v_{j+1}.
\]
Consider the moment when Algorithm~\ref{alg:greedy_matching} selects the first vertex from $C_i$.
Suppose the first such vertex is $u_j$.
At this time all $u_1,\ldots,u_k$ are unmatched, and all $v_1,\ldots,v_k$ are also unmatched, so both candidate pairs
\[
(u_j,v_{j+1}) = (u_j,\hat{\pi}(u_j))
\qquad\text{and}\qquad
(u_j,v_j) = (u_j,\pi^*(u_j))
\]
are feasible for matching.

Since the greedy algorithm processes edges in nonincreasing order of scores and chooses $(u_j,\hat{\pi}(u_j))$ at this moment, we must have
\begin{equation}\label{lower_score}
    \Score(u_j,\hat{\pi}(u_j)) \geq \Score(u_j,\pi^*(u_j)).
\end{equation}

Because cycles are disjoint, this argument applies independently to each $C_i$, yielding vertices $u_i\in C_i$ satisfying~\eqref{lower_score}.

Let $\Pi^\ell$ denote the set of estimators $\hat{\pi}$ such that $\hat{\pi}\circ(\pi^*)^{-1}$ has $\ell$ many distinct cycles. From the argument above, the event $\{\hat{\pi}\in \Pi^\ell\}$ implies the existence of distinct vertices
\[
u_1,\ldots,u_\ell\in\U
\]
such that
\[
\Score(u_i,\hat{\pi}(u_i)) \geq \Score(u_i,\pi^*(u_i)), \qquad i=1,\ldots,\ell.
\]

By Lemma \ref{lem:bad}, for each fixed $u\in\U$ and $\eps' > 0$, the score comparison above has probability
\begin{align*}
&\Prob{\Score(u,\hat{\pi}(u)) \geq \Score(u,\pi^*(u))}\\
& \leq \Prob{\Score(u,\hat{\pi}(u)) \geq \eps'\log n}+\Prob{\Score(u,\pi^*(u)) < \eps'\log n} \\
&= n^{-\eps'\log n + o(\eps' \log n)}+n^{-\lambda s^2 + o(\eps')}\\
&=
n^{-\lambda s^2 + o(\eps')}
\end{align*}
Taking a sequence of $\eps' \rightarrow 0$, we then get an upper bound of $n^{-\lambda s^2 + o(1)}$.

Next, for a fixed estimator $\hat{\pi}$ with $\ell$ cycles, the scores of different mismatched  $u_i$ involve disjoint vertex sets and are conditionally independent given $\R$. As a result, for any distinct $u_1,\ldots,u_\ell$,
\[
\Prob{\bigcap_{i=1}^\ell
\{\Score(u_i,\hat{\pi}(u_i)) \geq \Score(u_i,\pi^*(u_i))\}}
\leq
n^{-\ell\lambda s^2 + o(\ell)}.
\]

There are at most $|\U|^\ell = n^{\ell(1-\alpha)}$ possible $\ell$--tuples of distinct vertices.
Summing over them,
\begin{align*}
\Prob{\hat{\pi}\in\Pi^\ell}
&\leq n^{\ell(1-\alpha)} \cdot n^{-\ell\lambda s^2 + o(\ell)} = n^{-\ell(\lambda s^2 - (1-\alpha)) + o(\ell)}.
\end{align*}
Using the assumed gap $\lambda s^2 > (1-\alpha)+\eps$, this gives
\[
\Prob{\hat{\pi}\in\Pi^\ell}
\leq n^{-\ell\eps + o(\ell)}.
\]

Summing over all possible numbers of cycles:
\begin{align*}
\Prob{\hat{\pi}\neq\pi^*}
&= \sum_{\ell=1}^{|\U|/2} \Prob{\hat{\pi}\in\Pi^\ell} \leq \sum_{\ell=1}^{|\U|/2} n^{-\ell\eps + o(\ell)} = n^{-\eps + o(1)} = o(1).
\end{align*}

Therefore, Algorithm~\ref{alg:greedy_matching} recovers $\pi^*$ with probability $1-o(1)$.
\end{proof}

We show a lemma separating the objective function in Algorithm \ref{alg:lp_matching} into three sets of terms: edges between revealed nodes, edges across revealed and unrevealed nodes, and edges between unrevealed nodes.

\begin{lemma}\label{lem:lp_assignment_equivalence}
Let $(A,B)\sim \scsbm(n,a,b,s,\R,\sigma^*,\pi^*)$ and consider Algorithm~\ref{alg:lp_matching}. Then the objective can be rewritten as
\[
\|AD-DB\|_1
=
C - \Psi(D_{\U,\U})  + \Phi(D_{\U,\U}),
\]

where $D_{\U,\U}$ is restricted to the unrevealed nodes $\;\U$, $\Phi(D_{\U,\U})
=\sum_{u,v\in\U}\Bigl|(AD-DB)_{uv}\Bigr|$,
$\Psi(D_{\U,\U})
=2\sum_{u,v\in\U}D_{u v}\Score(u,v)$, $C > 0$ is a constant and $\Pi_{\U}$ denotes permutations on $[n]$ that fix $\R$.
\end{lemma}
\begin{proof}
For any feasible $(D,Y)$, replacing $Y$ by $|AD-DB|$ entrywise preserves feasibility and does not increase the objective, so the LP objective equals
\[
f(D)=\|AD-DB\|_1.
\]
We partition indices into $\R$ and $\U$. Since $D_{r,\pi_\R(r)}=1$ for all $r\in\R$, terms involving only rows/columns in $\R$ are constant in $D_{\U,\U}$. Thus
\[
f(D)=C_0 + \sum_{u\in\U}\sum_{r\in\R}
\Bigl|A_{u r}-\sum_{v\in\U}D_{u v}B_{v,\pi_\R(r)}\Bigr|
+\Phi(D_{\U,\U}),
\]
where $C_0$ is constant and $\Phi(\cdot)$ collects the terms indexed only by $\U\times\U$.

Fix $u\in\U$ and $r\in\R$. Since $A_{ur}\in\{0,1\}$, for any $x\in[0,1]$ we have
\[
|A_{u r}-x| = A_{ur} + x -2A_{ur}x.
\]
Therefore
\begin{align*}
&\sum_{r\in\R}\Bigl|A_{u r}-\sum_{v\in\U}D_{u v}B_{v,\pi_\R(r)}\Bigr| = \deg_{A,\R}(u) +\sum_{v\in\U}D_{u v}\left(\deg_{B,\pi_\R}(v)
-2\Score(u,v)\right),
\end{align*}

and summing over $u\in\U$ gives
\[
\sum_{u\in\U}\sum_{r\in\R}
\Bigl|A_{u r}-\sum_{v\in\U}D_{u v}B_{v,\pi_\R(r)}\Bigr|
=
C_1 -2\sum_{u,v\in\U}D_{uv}\Score(u,v),
\]
with $C_1$ constant. Hence
\[
f(D)= C - \Psi(D_{\U,\U})  + \Phi(D_{\U,\U}),
\]
for a constant $C$, where $\Phi(D_{\U,\U})
=\sum_{u,v\in\U}\Bigl|(AD-DB)_{uv}\Bigr|$ and
$\Psi(D_{\U,\U})
=2\sum_{u,v\in\U}D_{u v}\Score(u,v)$.
\end{proof}

We now show a concentration result on the unrevealed component of the objective in Algorithm \ref{alg:lp_matching}.

\begin{lemma}\label{lem:unrevealed_concentration}
Let $D$ be a right--stochastic matrix, $D^*$ be the permutation corresponding to the true matching for $A,B$, and let $\Phi(D):= \sum_{u,v\in\U}\Bigl|(AD-DB)_{uv}\Bigr|$. We now select an arbitrary vertex $x \in \U$ and define $D^x(x, \cdot) = D^*(x,\cdot)$ and $D^x = D$ otherwise. Here $D^x$ is obtained from $D$ by replacing the $x$-th row with the
corresponding row of the true permutation matrix $D^*$. In addition, let $d,d^*$ be the rows of $x$ in $D,D^x$ respectively. Then with probability 1-o(1):
\[
\frac{|\Phi(D^x)-\Phi(D)|}{\|d-d^*\|_1} \leq C\left(\ddegA(x) + \max_v \ddegB(v)\right)
\]
for some constant $C > 0$.
\end{lemma}

\begin{proof}
We first expand $\Phi(D),\Phi(D^x)$ to get:
\begin{align*}
\Phi(D) &= \sum_{u,v \in \U} |(AD-DB)_{uv}|= \sum_{u,v \in \U} \Bigl| \sum_{w } (A_{uw}D_{wv} - D_{uw}B_{wv})\Bigr|.    
\end{align*}

\begin{align*}
\Phi(D^x) &=  \sum_{u,v \in \U} |AD^x-D^xB|_{uv} = \sum_{u,v \in \U} \Bigl| \sum_{w } (A_{uw}D^x_{wv} - D^x_{uw}B_{wv})\Bigr|.    
\end{align*}

We then split $ \Phi(D)- \Phi(D^x)$ into distinct rowwise cases of $u = x$ and $u \neq x$.
\begin{align*}
    \Phi(D)- \Phi(D^x)  &=  \sum_{u,v \in \U} \left(|(AD-DB)_{uv}| -|AD^x-D^xB|_{uv}\right) \\
    &= \sum_{u \neq x, v \in \U} \left(|(AD-DB)_{uv}| -|AD^x-D^xB|_{uv}\right) \\
    &\quad + \sum_{u = x, v \in \U} \left(|(AD-DB)_{uv}| -|AD^x-D^xB|_{uv}\right)
\end{align*}

For $u\neq x$, we note that
\begin{align*}
&\sum_{u \neq x}\sum_{v \in \U}
\Bigl(|(AD-DB)_{uv}|-| (AD^x-D^xB)_{uv}|\Bigr) \\
&= \sum_{u \neq x}\sum_{v \in \U}
\Bigl(|(AD)_{uv}-(DB)_{uv}|
      -|(AD^x)_{uv}-(D^xB)_{uv}|\Bigr) \\
&\le \sum_{u \neq x}\sum_{v \in \U}
\Bigl(|(AD)_{uv}-(AD^x)_{uv}|
      +|(DB)_{uv}-(D^xB)_{uv}|\Bigr) \\
&= \sum_{u \neq x}\sum_{v \in \U}
\Bigl(|(A(D-D^x))_{uv}|
      +|((D-D^x)B)_{uv}|\Bigr) \\
&= \sum_{u \neq x}\sum_{v \in \U}
\Bigl(\Bigl|\sum_{w} A_{uw}(D-D^x)_{wv}\Bigr| +\Bigl|\sum_{w} (D-D^x)_{uw}B_{wv}\Bigr|\Bigr) \\
&= \sum_{u \neq x}\sum_{v \in \U}
\Bigl(|A_{ux}(D-D^x)_{xv}|
      +|(D-D^x)_{ux}B_{xv}|\Bigr) \\
&\le \sum_{u \neq x}\sum_{v \in \U}
A_{ux}\,|D_{xv}-D^x_{xv}| \\
&\le \sum_{u \in \U} A_{ux}\,\|d-d^*\|_1 \\
&= \ddegA(x)\,\|d-d^*\|_1 .
\end{align*}

Then for $u=x$, we get that
\begin{align*}
&\sum_{v \in \U}
\Bigl(|(AD-DB)_{xv}|-| (AD^x-D^xB)_{xv}|\Bigr) \\
&\le \sum_{v \in \U}
\Bigl(|(AD)_{xv}-(AD^x)_{xv}|
      +|(DB)_{xv}-(D^xB)_{xv}|\Bigr) \\
&= \sum_{v \in \U}
\Bigl(|(A(D-D^x))_{xv}|
      +|((D-D^x)B)_{xv}|\Bigr) \\
&= \sum_{v \in \U}
\Bigl(\Bigl|\sum_{w} A_{xw}(D-D^x)_{wv}\Bigr| +\Bigl|\sum_{w} (D-D^x)_{xw}B_{wv}\Bigr|\Bigr) \\
&\le \sum_{v \in \U}\sum_{w} A_{xw}\,|(D-D^x)_{wv}|
   +  \sum_{v \in \U}\sum_{w} |(D-D^x)_{xw}|\,B_{wv} \\
&= \sum_{v \in \U} A_{xx}\,|D_{xv}-D^x_{xv}|
   +  \sum_{w} |D_{xw}-D^x_{xw}|\sum_{v \in \U} B_{wv} \\
&= A_{xx}\,\|d-d^*\|_1
   +  \sum_{w} |D_{xw}-D^x_{xw}|\,\ddegB(w) \\
&\le A_{xx}\,\|d-d^*\|_1
   +  \|d-d^*\|_1 \max_{w}\ddegB(w) \\
&\le \|d-d^*\|_1\bigl(1+\max_{w}\ddegB(w)\bigr).
\end{align*}

As $\ddegA(u), \ddegB(u) = \Theta(|\U|a\log n/n) = \Theta(n^{-\alpha}\log n)$ with probability $1-o(1)$, we can choose $C = 2$ on the degree-concentration event to get
\[
\frac{|\Phi(D^x)-\Phi(D)|}{\|d-d^*\|_1} \leq C \left(\ddegA(x) + \max_v \ddegB(v)\right).
\]
as desired.
\end{proof}

Next, we show a generalization of Lemma \ref{lem:bad} comparing convex combinations of scores.

\begin{lemma}\label{lem:score_gap_lp} Let $(A,B) \sim \scsbm(n,a,b,s,\R,\sigma^*,\pi^*)$ be defined with $|\U| = n^{1-\alpha}$ for some $\alpha \in [0,1]$. Fix $x\in\U$ and write $v^*=\pi^*(x)$.
For any probability vector $d=(d_v)_{v\in\U}$ supported on $\U\setminus\{v^*\}$, define
\[
S_d(x):=\sum_{v\in\U} d_v \Score(x,v).
\]
Then there exists a constant $c_0>0$ such that
\[
\Prob{S_d(x)\leq \Score(x,v^*) - c_0\log n}
\geq 1-n^{-\lambda s^2+o(1)}
\]
uniformly over all such $d$.
Consequently, if $\lambda s^2 > 1-\alpha$, the inequality above holds simultaneously for all $x\in\U$ and all such $d$ with probability $1-o(1)$.
\end{lemma}

\begin{proof}
Fix $x\in\U$ and write $v^*=\pi^*(x)$.  Choose constants
\[
0<\eps_1<\eps_2<\lambda s^2,
\]
to be specified below, and consider the two events
\begin{align*}
&E_{\mathrm{true}}(x):=\{\Score(x,v^*)\ge \eps_2 \log n\}, \\
&E_{\mathrm{false}}(x):=\Big\{\max_{v\in\U\setminus\{v^*\}}\Score(x,v)\le \eps_1 \log n\Big\}.
\end{align*}

By Lemma~\ref{lem:bad} for the true matching $v^*$ with $\eps=\eps_2$,
\[
\Prob{E_{\mathrm{true}}(x)^c}
=\Prob{\Score(x,v^*)\le \eps_2\log n}
\le n^{-\lambda s^2+O(\eps_2)}.
\]

For any false match $v\neq v^*$, Lemma~\ref{lem:bad} with $\eps=\eps_1$ gives
\[
\Prob{\Score(x,v)\ge \eps_1\log n}
\le n^{-\eps_1\log n+o(\eps_1\log n)}.
\]
Taking a union bound over the at most $|\U|\le n^{1-\alpha}$ false candidates yields
\begin{align*}
\Prob{E_{\mathrm{false}}(x)^c}
&\le \sum_{v\in\U\setminus\{v^*\}} \Prob{\Score(x,v)\ge \eps_1\log n}\\
&\le |\U|\cdot n^{-\eps_1\log n+o(\eps_1\log n)} \\
&= n^{1-\alpha-\eps_1\log n+o(\eps_1\log n)}
\end{align*}

On the event $E_{\mathrm{true}}(x)\cap E_{\mathrm{false}}(x)$, for any probability vector $d=(d_v)_{v\in\U}$ supported on $\U\setminus\{v^*\}$ we have
\[
S_d(x)=\sum_{v\in\U} d_v \Score(x,v)\le \max_{v\neq v^*}\Score(x,v)\le \eps_1\log n,
\]
and simultaneously $\Score(x,v^*)\ge \eps_2\log n$. Hence
\[
S_d(x)\le \Score(x,v^*)-(\eps_2-\eps_1)\log n.
\]
Therefore the desired inequality holds with $c_0:=\eps_2-\eps_1>0$, and moreover
\begin{align*}   
&\Prob{S_d(x)\le \Score(x,v^*)-c_0\log n} \\
&\qquad \ge 1-\Prob{E_{\mathrm{true}}(x)^c}-\Prob{E_{\mathrm{false}}(x)^c} \\
&\qquad=1-n^{-\lambda s^2+o(1)},
\end{align*}

uniformly over all such $d$.

Finally, a union bound over $x\in\U$ gives a failure probability at most
\[
|\U|\cdot n^{-\lambda s^2+o(1)}
= n^{1-\alpha-\lambda s^2+o(1)},
\]
which is $o(1)$ when $\lambda s^2>1-\alpha$.  (For concreteness one may take $\eps_2=\tfrac34\lambda s^2$ and $\eps_1=\tfrac12\lambda s^2$, so $c_0=\tfrac14\lambda s^2$.)
\end{proof}

We use the lemmas above to show that Algorithm \ref{alg:lp_matching} achieves the threshold for exact recovery.

\begin{theorem}\label{th:lp_SBM}
Let $(A,B) \sim \scsbm(n,a,b,s,\R,\sigma^*,\pi^*)$ be defined  with $|\U| = n^{1-\alpha}$ for some $\alpha \in [0,1]$. Then setting $\lambda = \frac{a+b}{2}$, if $\lambda s^2 > (1-\alpha) + \eps$ for some $\eps > 0$, Algorithm \ref{alg:lp_matching} returns an estimator $\hat{\pi}$ in polynomial time such that
\[
\Prob{\hat{\pi} = \pi^*} = 1 -o(1).
\]
\end{theorem}

\begin{proof}

By Lemma~\ref{lem:lp_assignment_equivalence}, for any feasible $D$ with the seed rows fixed the LP objective is
\[
f(D)
:= \|AD-DB\|_1
= C - \Psi(D_{\U,\U}) + \Phi(D_{\U,\U}),
\]
where $C>0$ is a constant independent of $D_{\U,\U}$,
\begin{align*}
\Psi(D_{\U,\U}) = 2\sum_{u,v\in\U} D_{uv}\Score(u,v), \qquad \Phi(D_{\U,\U}) = \sum_{u,v\in\U} \bigl|(AD-DB)_{uv}\bigr|.  
\end{align*}

So minimizing $f(D)$ is equivalent to balancing the revealed-unrevealed interaction term $-\Psi$ (which prefers large scores) against the within-unrevealed interaction term $\Phi$.

Let $D^*$ denote the permutation matrix associated with the true matching $\pi^*$ (with the seed rows fixed). Our goal is to show that, with high probability, every minimizer $D$ of $f$ has $D_{\U,\U}=D^*_{\U,\U}$; then Algorithm~\ref{alg:lp_matching} recovers $\pi^*$.

Fix any feasible bistochastic $D$ (respecting the seed constraints) and any vertex $x\in\U$. Let
\[
d := D_{x,\cdot}, \qquad d^* := D^*_{x,\cdot} = e_{\pi^*(x)}.
\]
Define $D^x$ to be the matrix obtained from $D$ by replacing the $x$-th row with $d^*$.Note that $D^x$ need not be feasible, since the column-sum constraints on $\U$ may be violated. 

We compare $f(D)$ to $f(D^x)$ for each $x\in\U$ and show that whenever $d\neq d^*$, then $f(D)>f(D^x)$. Iterating this comparison over all $x\in\U$ yields a chain of inequalities
\[
f(D) > f(D^{x_1}) > f(D^{x_2}) > \cdots > f(D^*),
\]
where $D^*$ is the true permutation matrix. Thus, although the intermediate matrices $D^{x_i}$ need not be feasible, they serve only to certify strict descent of the objective toward the feasible point $D^*$. A direct computation using Lemma~\ref{lem:lp_assignment_equivalence} gives
\begin{align*}
f(D) - f(D^x)
&= \bigl[C - \Psi(D_{\U,\U}) + \Phi(D_{\U,\U})\bigr] - \bigl[C - \Psi(D^x_{\U,\U}) + \Phi(D^x_{\U,\U})\bigr] \\
&= \bigl[\Psi(D^x_{\U,\U}) - \Psi(D_{\U,\U})\bigr] + \bigl[\Phi(D_{\U,\U}) - \Phi(D^x_{\U,\U})\bigr].
\end{align*}
We now bound the $\Psi$ and $\Phi$ terms separately.

By the definition of $\Psi$,
\begin{align*}
\Psi(D^x_{\U,\U}) - \Psi(D_{\U,\U})
&= 2\sum_{v\in\U} (D^x_{xv} - D_{xv}) \Score(x,v) = 2\Bigl[\Score(x,\pi^*(x)) - \sum_{v\in\U} d_v \Score(x,v)\Bigr].
\end{align*}
Conditioned on the seeds and their images, $\Score(x,\pi^*(x))$ is a sum of $|\R|$ independent Bernoulli variables with mean of order $\lambda s^2$ per seed, and each $\Score(x,v)$ for $v\neq\pi^*(x)$ is stochastically smaller, exactly as in the proof of Theorem~\ref{th:upper_SBM}. As $\lambda s^2 > 1-\alpha + \eps$, by Lemma~\ref{lem:score_gap_lp}, there exists a constant $c_0>0$ such that with probability $1-o(1)$,
\[
\sum_{v\in\U} d_v \Score(x,v)
\leq\Score(x,\pi^*(x)) - c_0 \log n
\]
simultaneously for all probability vectors $d$ supported on $\U\setminus\{\pi^*(x)\}$ and all $x\in\U$. On this event we obtain
\begin{equation}\label{eq:psi-gap}
\Psi(D^x_{\U,\U}) - \Psi(D_{\U,\U})
\geq 2c_0 \log n \|d-d^*\|_1.
\end{equation}

By Lemma~\ref{lem:unrevealed_concentration},
\[
\frac{\bigl|\Phi(D^x_{\U,\U}) - \Phi(D_{\U,\U})\bigr|}{\|d-d^*\|_1}
\leq C\Bigl(\ddegA(x) + \max_{v\in\U}\ddegB(v)\Bigr)
\]
for some constant $C>0$, with probability $1-o(1)$ (uniformly over all $x$ and all right--stochastic $D$).
Standard degree concentration in the sparse SBM ensures that there exists $c_1>0$ such that, with probability $1-o(1)$,
\[
\ddegA(x) + \max_{v\in\U}\ddegB(v)
\leq c_1 n^{-\alpha}\log n
\quad\text{for all }x\in\U.
\]
We now justify the (loose) upper bound on the maximum degree term, allowing slack of order
$n^{\alpha/2}\log n$.

Fix $u\in\U$ and write
\[
d_A(u):=\sum_{v\in[n]}A(u,v).
\]
Since $A$ is an SBM with parameters $(a,b)$, we have
\begin{align*}
&\Ex{d_A(u)}=\tfrac{a+b}{2}\log n=\lambda\log n,\quad \sum_{v\in[n]}\Var{A(u,v)} \le \Ex{d_A(u)}=\lambda\log n,
\end{align*}
where $\lambda=(a+b)/2$.

By Bennett's inequality, for any $x>0$ and $h(x)=(1+x)\log(1+x)-x$,
\begin{align*}
\Prob{d_A(u)\ge \Ex{d_A(u)}+t} &\quad \le \exp\!\left(
-\Big(\sum_{v\in[n]}\Var{A(u,v)}\Big)\,
h\!\left(\frac{t}{\sum_{v\in[n]}\Var{A(u,v)}}\right)
\right)\\
&\quad \le \exp\!\left(
-\lambda\log n\cdot h\!\left(\frac{t}{\lambda\log n}\right)
\right).
\end{align*}
We then take $t := n^{\alpha/2}\log n$. Using the asymptotic $h(x)\ge \tfrac12 x\log(1+x)$ for all sufficiently large $x$,
we obtain for all large $n$,
\begin{align*}
\lambda\log n\cdot h\!\left(\frac{t}{\lambda\log n}\right) 
&\geq
\lambda\log n \cdot \frac12\cdot \frac{n^{\alpha/2}}{\lambda}\,
\log\!\Bigl(1+\frac{n^{\alpha/2}}{\lambda}\Bigr) \\
&\ge
c\, n^{\alpha/2}(\log n)^2
\end{align*}
for some constant $c>0$. Hence for each fixed $u\in\U$,
\begin{align*}
\Prob{d_A(u)\ge (\lambda + n^{\alpha/2})\log n}
\le \exp\big(-cn^{\alpha/2}(\log n)^2\big).
\end{align*}
Applying a union bound over $u\in\U$ (with $|\U|=n^{1-\alpha}$) yields
\begin{align*}
&\Prob{\max_{u\in\U} d_A(u)\ge \lambda\log n + n^{\alpha/2}\log n} \\
&\quad\le
n^{1-\alpha}\exp\!\big(-c\,n^{\alpha/2}(\log n)^2\big)
=o(1).
\end{align*}
The same bound holds for $B$, and therefore with probability $1-o(1)$,
\begin{align*}
&\max_{u\in\U} d_A(u)\le (\lambda+n^{\alpha/2})\log n \\
&
\max_{v\in\U} d_B(v)\le (\lambda+n^{\alpha/2})\log n.
\end{align*}

Recalling that $\ddegA(x)$ and $\ddegB(v)$ denote degrees into $\U$ normalized by $|\R|=n^{1-\alpha}$,
we have $\ddegA(x)\le d_A(x)/|\R|$ and $\ddegB(v)\le d_B(v)/|\R|$, so on the same event, for all $x \in \U$,
\begin{align*}
\ddegA(x) + \max_{v\in\U}\ddegB(v)
&\le
\frac{2(\lambda+n^{\alpha/2})\log n}{|\R|}\\
&=
2(\lambda+n^{\alpha/2})\,n^{\alpha-1}\log n.
\end{align*}
Equivalently, setting $c_1:=2(\lambda+1)$ and using $n^{\alpha/2}\le n^\alpha$ for $\alpha\in[0,1]$,
we may write the convenient slack form
\[
\ddegA(x) + \max_{v\in\U}\ddegB(v)
\le
c_1\,n^{-\alpha/2}\log n
\]
holding for all $x\in \U$ with probability $1-o(1)$. 
Intersecting this event with the high-probability event from the previous step, we have
\begin{equation}\label{eq:phi-bound}
\bigl|\Phi(D^x_{\U,\U}) - \Phi(D_{\U,\U})\bigr|
\leq C c_1 n^{-\alpha/2}\log n \|d-d^*\|_1
\end{equation}
for all $x\in\U$ and all right--stochastic $D$, with probability $1-o(1)$.

Combining \eqref{eq:psi-gap} and \eqref{eq:phi-bound}, we obtain on a $1-o(1)$ event:
\begin{align*}
f(D) - f(D^x) &= \bigl[\Psi(D^x_{\U,\U}) - \Psi(D_{\U,\U})\bigr] + \bigl[\Phi(D_{\U,\U}) - \Phi(D^x_{\U,\U})\bigr] \\
&\geq 2c_0 \log n \|d-d^*\|_1
   - C c_1 n^{-\alpha/2}\log n \|d-d^*\|_1 
\end{align*}
Then with probability $1-o(1)$, taking $n$ sufficiently large, we have the strict inequality
\[
f(D) > f(D^x)
\quad\text{whenever } d\neq d^*.
\]
That is, for each row $x\in\U$, we can strictly decrease the objective by replacing the $x$-th row of $D$ with the true permutation row $d^*$, unless it was already equal to $d^*$.

Iterating this argument over all $x\in\U$, we see that with probability $1-o(1)$ any minimizer $D$ of $f$ must satisfy $D_{x,\cdot} = d^*$ for all $x\in\U$, i.e.,
\[
D_{\U,\U} = D^*_{\U,\U}
\]
is the unique optimal choice for the unrevealed block.

We have shown that, with probability $1-o(1)$, every optimal LP solution satisfies $D_{\U,\U} = D^*_{\U,\U}$, the permutation matrix of $\pi^*$ restricted to $\U$. In particular, the solution $(D^*,Y^*)$ returned by the LP solver has this property with high probability. In the projection step of Algorithm~\ref{alg:lp_matching} we run the Hungarian algorithm on $W := D^*_{\U,\U}$ and obtain a permutation $\pi_{\U}$. Since $W$ itself is a permutation matrix, any maximum-weight matching w.r.t.\ $W$ is exactly the permutation encoded by $W$, so $\pi_{\U} = \pi^*|_{\U}$.

Finally, we set $\hat{\pi}(u)=\pi_{\U}(u)$ for $u\in\U$, so $\hat{\pi}=\pi^*$ with probability $1-o(1)$, as claimed.

\end{proof}

\section{Experiments and Discussion}

We empirically evaluate the algorithms in Algorithms \ref{alg:seeded_matching}--\ref{alg:lp_matching} in the almost fully seeded regime. The experiments are designed to test two claims suggested by our theory:
\begin{enumerate}
    \item For $\scsbm$ instances with $|\U| = n^{1-\alpha}$, accuracy should approach exact recovery as $s$ increases such that $\lambda s^2 > 1-\alpha$.
    \item On real graphs, simple seed-neighborhood overlap methods should remain competitive with established baselines on a range of problems from the matching literature.
\end{enumerate}

\subsection{Experimental Protocol}
Given the ground truth $\pi^*$, we report the unseeded accuracy
\[
\mathrm{Acc}(\hat{\pi})=\frac{1}{|\U|}\sum_{u\in\U}\Ind{\hat{\pi}(u)=\pi^*(u)}.
\]
We additionally report runtime in seconds for the matching stage on $\U$. The reported numbers are means over 20 trials, with one standard deviation in parentheses.

For each trial and dataset, we select a revealed set $\R$ uniformly at random among vertices, with a prescribed seed fraction $\delta = |\R|/n$, then we evaluate accuracy on $\U = [n] \setminus \R$ only.

For the synthetic plots in Figure~\ref{fig:accuracy-seedfrac} we vary the seed fraction $\delta$; for the dataset summary tables we fix $\delta=0.9$, except for the auto dataset, where $\delta = 0.99$.

\subsection{Algorithms}
We now specify the methods used in the experiments. As described in Algorithms \ref{alg:seeded_matching} and  \ref{alg:greedy_matching}, the Hungarian and Greedy methods use the score matrix
$\Score(u,v)=|\Nbr[A]{u}\cap \Nbr[B]{v}|$, where neighborhoods are restricted to revealed nodes $\R$.
\textsc{SGM} is the seeded Frank--Wolfe QAP approach of \cite{fishkind2019seeded} (via \texttt{graspologic}). SGM applies a Frank--Wolfe relaxation to the quadratic assignment problem with seed constraints and is among the strongest scalable baselines for seeded graph matching on real networks. Unlike our methods, SGM optimizes a quadratic objective rather than a linear overlap-based criterion.
Our two-hop baseline \textsc{Hop2} follows the two-hop signature idea in \cite{ding2021efficient} and subsequent uses in \cite{yu2021graph,mao2023}, using 2-hop neighborhoods from unrevealed nodes. Nodes are represented by binary two-hop connectivity patterns, and matching is performed via neighborhood overlap. The \textsc{FW-Linear} baseline is a fast Frank--Wolfe heuristic for the $\ell_1$ objective $\min_{D\in\mathcal{D}}\|AD-DB\|_1$ with seed constraints as specified in Algorithm \ref{alg:fw_linear_matching}. Due to runtime concerns, we run the original LP \textsc{Exact LP} from Algorithm \ref{alg:lp_matching} in a separate comparison with \textsc{FW-Linear}.

\subsection{Datasets}
We follow standard preprocessing, applying symmetrization to all adjacency matrices as in \cite{Lyzinski14,fishkind2019seeded,ding2021efficient}. All graphs are treated as unweighted and undirected.

\begin{enumerate}[leftmargin=*,itemsep=2pt]
\item \textbf{Synthetic CSBM:} We generate $(A,B)\sim\scsbm(n,a,b,s,\R,\sigma^*,\pi^*)$ with balanced communities and connection probabilities $a\log n/n$ and $b\log n/n$, using $a=5$, $b=1$ and $s=0.8$. We set $n=1000$ for the table, and vary $\delta$ for Figure~\ref{fig:accuracy-seedfrac}.
\item \textbf{C.\ elegans connectome:} Two modalities (chemical and electrical) on a common vertex set, following \cite{Lyzinski14,fishkind2019seeded}.
\item \textbf{Enron email:} Communication graph used in seeded matching benchmarks \cite{Lyzinski14,fishkind2019seeded}.
\item \textbf{Wikipedia bilingual neighborhoods:} Matching between French and English language neighborhoods as in \cite{Lyzinski14,fishkind2019seeded}.
\item \textbf{Facebook friendship network:} A friendship network of Stanford students and faculty originally used in \cite{Traud2012} and further studied in \cite{ding2021efficient,yu2021graph}.
\item \textbf{Autonomous systems (AS):} Multiple snapshots of Oregon Route Views peering graphs, as used in \cite{yu2021graph}. We report results for a representative pair along the first 10,000 nodes. For this dataset, we use a seed fraction of $\delta = 0.99$.
\end{enumerate}

\subsection{Results}
Tables~\ref{tab:accuracy}--\ref{tab:runtime} summarize accuracy and runtime at the seed fraction $\delta=0.9$. Table \ref{tab:accuracy-linear} compares the Frank--Wolfe approximation to the underlying linear program in Algorithm \ref{alg:lp_matching}. In addition to keeping track of the accuracy and runtime, we include the size of the largest connected component of the intersection graph. For this table, the correlated SBM was reduced to 500 nodes and the autonomous systems dataset was excluded for runtime considerations. Figure~\ref{fig:accuracy-seedfrac} shows the synthetic CSBM accuracy as a function of the seed fraction.

\begin{table*}[ht]
\centering
\small
\begin{tabular}{lllllll}
\hline
Dataset & Hungarian & Greedy & SGM & Hop2 & FW-Linear & Size \\
\hline
csbm & 0.87 (0.03) & 0.88 (0.03) & 0.90 (0.04) & 0.86 (0.03) & 0.90 (0.04) & 1000 \\
elegans & 1.00 (0.00) & 1.00 (0.00) & 1.00 (0.00) & 1.00 (0.00) & 1.00 (0.00) & 279 \\
wiki & 0.83 (0.04) & 0.76 (0.06) & 0.84 (0.04) & 0.49 (0.05) & 0.85 (0.04) & 1382 \\
enron & 0.55 (0.11) & 0.52 (0.10) & 0.83 (0.12) & 0.48 (0.12) & 0.78 (0.08) & 184 \\
facebook & 1.00 (0.01) & 1.00 (0.00) & 1.00 (0.01) & 1.00 (0.01) & 1.00 (0.00) & 769 \\
auto & 0.69 (0.06) & 0.67 (0.06) & 0.88 (0.04) & 0.83 (0.05) & 0.87 (0.04) & 10000 \\
\hline
\end{tabular}
\caption{Mean accuracy by dataset over 20 trials (1 sd)}
\label{tab:accuracy}
\end{table*}

\begin{table*}[ht]
\centering
\small
\begin{tabular}{lllllll}
\hline
Dataset & Hungarian & Greedy & SGM & Hop2 & FW-Linear & Size \\
\hline
csbm & 0.07 (0.06) & 0.06 (0.01) & 0.31 (0.01) & 0.08 (0.01) & 0.10 (0.01) & 1000 \\
elegans & 0.02 (0.01) & 0.02 (0.00) & 0.03 (0.02) & 0.05 (0.10) & 0.03 (0.00) & 279 \\
wiki & 0.25 (0.11) & 0.21 (0.01) & 0.70 (0.18) & 0.29 (0.02) & 0.27 (0.02) & 1382 \\
enron & 0.00 (0.01) & 0.01 (0.01) & 0.03 (0.01) & 0.01 (0.01) & 0.01 (0.01) & 184 \\
facebook & 0.18 (0.01) & 0.22 (0.18) & 0.21 (0.02) & 0.25 (0.02) & 0.21 (0.02) & 769 \\
auto & 0.31 (0.22) & 0.29 (0.14) & 0.37 (0.15) & 1.75 (0.17) & 1.73 (0.15) & 10000 \\
\hline
\end{tabular}
\caption{Mean runtime by dataset over 20 trials in seconds (1 sd)}
\label{tab:runtime}
\end{table*}

Across synthetic and real datasets, several patterns emerge. On the CSBM, all methods achieve near-perfect accuracy once $\delta$ reflects the separation predicted by the theory. The LP (Exact LP) and FW-Linear methods use additional information from connections between unrevealed nodes and generally have better performance. Nevertheless, neighborhood-overlap methods remain competitive with state-of-the-art SGM on most datasets, often meeting or exceeding the accuracy despite being simpler and faster. This supports our claim that when a large fraction of seeds are available, direct seed neighborhoods capture the usable alignment signal.

\begin{table*}[ht]
\centering
\small
\begin{tabular}{llllll}
\hline
Dataset & Exact LP & FW-Linear & Giant Size & $\Delta$ Objective & Size \\
\hline
csbm & 0.93 (0.04) & 0.93 (0.04) & 441 & 38.2 (2.27) & 500 \\
elegans & 1.00 (0.00) & 1.00 (0.00) & 279 & 0 (0.00) & 279 \\
wiki & 0.86 (0.05) & 0.87 (0.05) & 739 & 287 (22.4) & 800 \\
enron & 0.53 (0.15) & 0.77 (0.13) & 145 & 11.9 (8.41) & 184 \\
facebook & 1.00 (0.00) & 1.00 (0.01) & 762 & 1020 (92.4) & 769 \\
\hline
\end{tabular}
\caption{Accuracy and objective (FW-Exact) comparison between Algorithm \ref{alg:lp_matching} and Algorithm \ref{alg:fw_linear_matching}  (1 sd)}
\label{tab:accuracy-linear}
\end{table*}

Table~\ref{tab:accuracy-linear} compares the exact LP solution of Algorithm~\ref{alg:lp_matching} with its Frank--Wolfe approximation. Despite using a looser heuristic, FW-Linear achieves nearly identical accuracy on all datasets. When the LP objective is uniquely minimized at the true permutation, the Frank--Wolfe approximation typically attains an objective value close to the optimum; in cases where the optimum is zero, FW-Linear may converge to a nearby but still correct extreme point.

In particular, when the LP objective is minimized uniquely at the true permutation, the FW approximation reliably converges to a similar solution, even with a modest number of iterations. The small objective gaps observed on some datasets suggest that FW-Linear often reaches a near-optimal extreme point without requiring full convergence of the relaxation.

\begin{figure}[ht]
    \centering
    \includegraphics[width=0.9\linewidth]{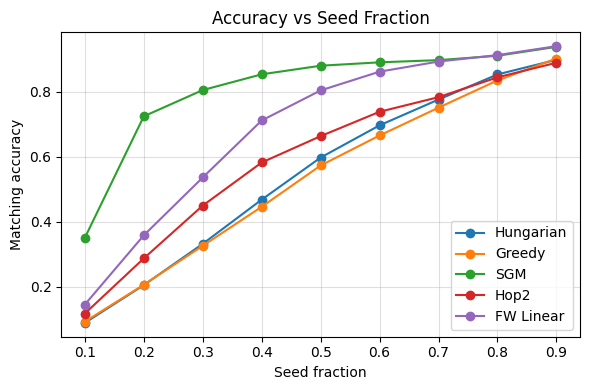}
    \caption{Accuracy as a function of seed fraction for CSBM on 1000 nodes.}
    \label{fig:accuracy-seedfrac}
\end{figure}

Figure ~\ref{fig:accuracy-seedfrac} shows the dependence of matching accuracy on the seed fraction $\delta$ averaged over twenty trials. We note that the Hungarian, greedy, and Frank--Wolfe LP curves exhibit a transition from partial to near-exact recovery as $\delta$ increases. In addition, the greedy and Hungarian methods perform similarly over the full range of $\delta$, suggesting that global optimality of the alignment is tied to the presence of strong local signals.

\subsection{Discussion}

Our results focus on the almost fully seeded regime for graph matching, where all but $n^{1-\alpha}$ vertices are revealed and the task is to align the remaining nodes efficiently. In this setting, we show that the presence of a large seed set fundamentally alters the difficulty of the problem: when $\lambda s^2 > 1-\alpha$, simple neighborhood--overlap statistics computed on the revealed vertices already contain enough signal to achieve exact recovery with high probability.
The synthetic experiments corroborate this theoretical picture, demonstrating that even simple neighborhood--overlap methods such as greedy or Hungarian matching perform competitively once the overlap signal is strong.

Beyond validating the theory, the experiments highlight a broader algorithmic message: in the almost fully seeded regime, the alignment signal is dominated by interactions with the revealed set, and global optimization over the unrevealed block becomes comparatively benign, as the objective landscape is dominated by the strong linear signal induced by the revealed vertices. This helps explain why simple score-based methods, a bistochastic LP relaxation, and its Frank--Wolfe approximation all behave similarly in practice when the seed fraction is large, despite their different computational profiles.

A compelling direction for future work is to extend our analysis to the joint community detection--matching problem in the semi-supervised correlated SBM. In this setting, one observes two correlated community-structured networks along with a subset of revealed vertex correspondences, and the objective is to simultaneously recover the true permutation and the community labels. Recent work has shown that matching and clustering can interact in subtle and mutually reinforcing ways: Gaudio, R\'acz, and Sridhar~\cite{Gaudio22} demonstrated that two correlated SBMs can enable exact recovery in regimes where neither matching nor community detection alone is feasible, while previous work~\cite{Nosratinia2018, Nis24} established that a vanishing fraction of revealed nodes can dramatically lower community-recovery thresholds. Wang et al.~\cite{Wang24} further showed that attribute information can shift feasibility boundaries in aligned inference tasks. Building on these insights, a natural goal is to characterize the optimal statistical and computational thresholds when both seed correspondences and partial community labels are available, and to design efficient algorithms that leverage both structural and labeling side information. Such an extension would unify seeded matching and semi-supervised clustering within a single theoretical framework, and delineate the precise trade-off between graph correlation, community signal strength, and the fraction of revealed nodes required for exact recovery in structured multi-graph settings.

\bibliographystyle{abbrvnat}
\bibliography{references}

\end{document}